\documentclass{amsart}
\usepackage{amsthm}
\usepackage{amssymb}
\usepackage{amsmath}
\usepackage{amsthm}

\usepackage[colorlinks,pagebackref=true,pdftex]{hyperref}
\usepackage[alphabetic,abbrev,backrefs,lite]{amsrefs} % for bibliography

\usepackage{lmodern} 
\usepackage[T1]{fontenc}

\theoremstyle{plain}
\newtheorem{theorem}{Theorem}[section]
\newtheorem{lemma}[theorem]{Lemma}
\newtheorem{corollary}[theorem]{Corollary}
\newtheorem{propn}[theorem]{Proposition}

\theoremstyle{definition}
\newtheorem{defn}[theorem]{Definition}
\newtheorem{notation}[theorem]{Notation}

\DeclareMathOperator{\Ext}{Ext}
\DeclareMathOperator{\homology}{H}
\def\loco{\homology}
\DeclareMathOperator{\supp}{Supp}
\DeclareMathOperator{\reg}{reg}
\def\totdeg#1{\Vert#1\Vert}

\def\naturals{\mathbb N}
\newcommand{\ints}{\mathbb{Z}}

\begin{document}
\title[Regularity of deficiency modules]{Regularity of Canonical and
Deficiency modules for Monomial ideals}

\author{Manoj Kummini}
\address{Department of Mathematics, Purdue University, West Lafayette, IN
47901 USA.}
\email{nkummini@math.purdue.edu}

\author{Satoshi Murai}
\address{Department of Mathematical Science, Faculty of Science, Yamaguchi
University, 1677-1 Yoshida, Yamaguchi 753-8512, Japan.} 


\begin{abstract}
We show that the Castelnuovo--Mumford regularity of the canonical or a
deficiency module of the quotient of a polynomial ring by a monomial ideal
is bounded by its dimension.
\end{abstract}

\maketitle

\section{Introduction}
Let $R = \Bbbk[x_1, \ldots, x_n]$ be a standard graded polynomial ring over
a field $\Bbbk$ and
$\mathfrak m = (x_1, \ldots, x_n)$ the homogeneous maximal ideal of $R$.
In this paper, we study the Castelnuovo--Mumford
regularity of the modules $\Ext^i_R(R/I, \omega_R)$ when $I\subset R$ is a
monomial ideal; here $\omega_R = R(-n)$ denotes the canonical module
of $R$. The $\Ext^i_R(R/I, \omega_R), i > n-\dim R/I$ are called
\emph{deficiency modules} of $R/I$ while $\Ext^{n-\dim R/I}_R(R/I,
\omega_R)$ is called the \emph{canonical module} of $R/I$.

For any homogeneous ideal $I \subseteq R$, local cohomology modules
$\loco^i_{\mathfrak m}(R/I)$ are important in commutative algebra and
algebraic geometry. One is often interested in the vanishing of homogeneous
components of $\loco^i_{\mathfrak m}(R/I)$. While one cannot expect the
vanishing of $\loco^i_{\mathfrak m}(R/I)$ in negative degrees, unless it
has finite length, one can, using the local duality theorem of
Grothendieck, obtain some information from $\Ext^{n-i}_R(R/I, \omega_R)$.
For a finitely generated graded $R$-module $M$, its
\emph{(Castelnuovo--Mumford) regularity}, $\reg(M)$, is an invariant that
contains information about the
stability of homogeneous components in sufficiently large degrees.  In
light of these, it is desirable to get bounds on $\reg\left( \Ext^i_R(R/I,
\omega_R)\right)$. Such bounds were studied by L.~T.~Hoa and
E.~Hyry~\cite{HoaHyryCMreg06} and M.~Chardin, D.~T.~Ha and Hoa
\cite{CHHregHdeg09}; see also the references in those papers.

Unfortunately, canonical and deficiency modules can have large regularity.
For a finitely generated graded $R$-module $M$, known bounds for
$\reg \left(\Ext^i_R(M, \omega_R)\right)$ are large (see,
\textit{e.g.},~\cite{HoaHyryCMreg06}*{Theorems 9 and 14}). On the other
hand, more optimal bounds for $\reg
\left(\Ext^i_R(R/I, \omega_R)\right)$ are known to exist for certain
classes of graded ideals $I$ (see \cite{HoaHyryCMreg06}*{Section 4}). It is
an interesting problem to find a class of graded ideals $I\subset R$ with
optimal bounds for $\reg\left( \Ext^i_R(R/I, \omega_R)\right)$. In this
paper, we focus on monomial ideals. It follows from the theory of
square-free modules, introduced by K.~Yanagawa \cite{YanaSRringsDuality00},
that if $I$ is a square-free monomial ideal then $\reg \left(\Ext^i_R(R/I,
\omega_R)\right) \leq \dim
\Ext^i_R(R/I, \omega_R)$. This bound is small, since $\dim \Ext^i_R(R/I,
\omega_R) \leq n-i$ (see~\cite{BrHe:CM}*{Corollary~3.5.11}).

While one cannot apply the theory of square-free modules to all monomial
ideals, there are results that show that, when $I$ is a monomial ideal,
$\reg \left(\Ext^i_R(R/I, \omega_R)\right)$ is not large. For example, we
see from~\cite{TakaGCM05}*{Proposition~1, p~333} that if $\Ext^i_R(R/I,
\omega_R)$ has finite length then its regularity is negative or equal to
zero. Again, Hoa
and Hyry~\cite{HoaHyryCMreg06}*{Proposition 21} showed that if
$\loco^i_{\mathfrak m}(R/I)$ has finite
length for $i=0,1,\dots,d-1$, where $d=\dim R/I$, then $\reg 
\left(\Ext^{n-d}_R(R/I, \omega_R)\right) \leq d$. We generalize these results in
the following theorem: 
\begin{theorem}
\label{thm:regIsLessThanDim}
Let $I \subseteq R$ be a monomial ideal. Then, for all $0 \leq i \leq n$,
\[
\reg \left(\Ext^i_R(R/I, \omega_R)\right) \leq \dim \Ext^i_R(R/I, \omega_R).
\]
\end{theorem}

Since $\dim \Ext^i_R(R/I, \omega_R) \leq n-i$
we immediately get:
\begin{corollary}
\label{thm:regIsLessThanNminusOne}
Let $I \subseteq R$ be a monomial ideal. Then, for all $0 \leq i \leq n$,
\[
\reg \left(\Ext^i_R(R/I, \omega_R)\right) 
\leq n-i.
\]
\end{corollary}

The above conclusion need not hold, in general, without the assumption that
$I$ is a monomial ideal;
see~\cite{ChardDCruzRegCurvesSurf03}*{Example 3.5}.

Our approach to bounding the regularity of canonical and deficiency modules
differs from that of Hoa and Hyry. We show that if $I$ is a monomial
ideal, then $\Ext^i_R(R/I, \omega_R)$ has a multigraded filtration,
called \textit{Stanley filtration}, introduced by D.~Maclagan and
G.~G.~Smith \cite{MacSmiUnifBdsMultgrReg05}; the bound on regularity
follows from this filtration.

In the next section, we discuss some preliminaries on Stanley filtrations
and local cohomology. In Section~\ref{sec:regAndDim} we prove our
main result.

\section{Preliminaries}
\label{sec:stanleyFiltr}

Hereafter we take $R$-modules to be graded by $\ints^n$, giving $\deg x_i =
\mathbf e_i$,
the $i$th unit vector of $\ints^n$.
We call this the
\emph{multigrading} of $R$ and $R$-modules.

\begin{notation}
Let $\mathbf a = (a_1, \ldots, a_n) \in \ints^n$. Write $\mathbf x^{\mathbf
a} = \prod_{i=1}^n x_i^{a_i} \in \Bbbk[x_1^{\pm 1}, \ldots, x_n^{ \pm 1}]$.
We say that $\mathbf a$ is the \emph{degree} of $\mathbf x^{\mathbf a}$,
and write $\deg \mathbf x^{\mathbf a} = \mathbf a$.
Define $\supp(\mathbf a) = \{ i : a_i \neq 0\}$. Define $\mathbf a^+,
\mathbf a^- \in \naturals^n$ by the conditions $\mathbf a = \mathbf a^+ -
\mathbf a^-$ and $\supp(\mathbf a^+) \cap \supp(\mathbf a^-) =
\varnothing$. We write $\totdeg{\mathbf a}$ for $\sum_{i=1}^n a_i$, the
\emph{total degree} of $\mathbf a$ (and of the monomial $\mathbf x^{\mathbf
a}$).  We will say that $\mathbf a$ (equivalently $\mathbf x^{\mathbf a}$)
is \emph{square-free} if $a_i \in \{0, 1\}$ for all $i$.  Let $[n] = \{1,
\ldots, n\}$.  For $\Lambda \subseteq [n]$, we set $\mathbf e_\Lambda =
\sum_{i \in \Lambda} \mathbf e_i$ and abbreviate the (square-free) monomial
$\mathbf x^{\mathbf e_\Lambda}$ as $x_\Lambda$. The canonical module of $R$
is $\omega_R = R(-\mathbf e_{[n]})$.
\end{notation}

Let $M$ be a finitely generated multigraded $R$-module.  Let $m \in M$ be a
homogeneous element and let $G \subset \{x_1,\dots,x_n\}$ be a subset such
that $um \neq 0$ for all monomials $u \in \Bbbk[G]$. The $\Bbbk$-subspace
$\Bbbk[G]m$ of $M$ generated by all the $um$, where $u$ is a monomial in
$\Bbbk[G]$, is called a \emph{Stanley space}. A \emph{Stanley
decomposition} of $M$ is a finite set $\mathcal S$ of pairs $(m, G)$ of
homogeneous elements $m \in M$ and $G \subseteq \{x_1, \ldots, x_n\}$ such
that $\Bbbk[G]m$ is a Stanley space for all $(m,G) \in \mathcal S$
and \begin{equation}
\label{eqn:stanleyDecompDefn}
M =_\Bbbk \bigoplus_{(m,G) \in \mathcal S} \Bbbk[G]m
\end{equation}
(We use $=_\Bbbk$ to emphasize that the decomposition is only as vector
spaces.) Properties of such decompositions
have been widely studied; we follow the approach
of~\cite{MacSmiUnifBdsMultgrReg05}*{Section~3} where Stanley decompositions
were used to get bounds for multigraded regularity.
Following~\cite{MacSmiUnifBdsMultgrReg05}*{Definition~3.7},
we define a \emph{Stanley filtration} to be a Stanley decomposition with an
ordering of
pairs $\{(m_i, G_i): 1 \leq i \leq p\}$ such that,  for $j =1,2,\dots,p$,
\[
\left(\sum_{i=1}^j Rm_i\right) \Big / \left(\sum_{i=1}^{j-1} Rm_i\right) =
\Bbbk[G_j](-\deg m_j).
\]
as $R$-modules. Note, in this case, that
\begin{equation*}
\label{eqn:stanleyFiltrDefn}
0 \subseteq Rm_1 \subseteq \cdots \subseteq \sum_{i=1}^j Rm_i \subseteq
\cdots  \subseteq \sum_{i=1}^p Rm_i = M
\end{equation*}
is a prime filtration of $M$, as in \cite{eiscommalg}*{p.~93, Proposition~3.7}.

\begin{propn}
\label{thm:stanleyDecompAndFiltr}
Let $M$ be a multigraded $R$-module with a Stanley decomposition $\mathcal
S$ such that for all $(m, G) \in \mathcal S$, $(\deg m)^+$ is square-free
and $G = \supp((\deg m)^+)$. Then $\mathcal S$ gives a Stanley filtration.
Moreover $\reg M \leq \max \{\totdeg{\deg m} : (m,G) \in \mathcal S\}$.
\end{propn}

\begin{proof}
We order $\mathcal S  = \{(m_1,G_1),\dots,(m_p,G_p)\}$ so that
$\totdeg{\deg m_1} \geq \cdots \geq \totdeg{\deg m_p}$.
% for all $1 \leq i < j\leq p$, $|G_i| > |G_j|$, or, $|G_i| = |G_j|$ and $\deg m_i \geq \deg m_j$.
It follows from our hypothesis that 
\begin{equation}
\label{add1}
\mathrm{span}_\Bbbk \{ m_1,\dots,m_p\}=
\mathrm{span}_\Bbbk \{ m \in M: \supp((\deg m)^+) \mbox{ is square-free}\},
\end{equation}
where $\mathrm{span}_\Bbbk(V)$ denotes the $\Bbbk$-vector space spanned by
elements in $V$.  Write $M^{(j)}$ for
 $\sum_{i=1}^j Rm_i$.  We will now show, inductively on $j$, that  
{\def\theenumi{\Alph{enumi}}\begin{enumerate}
\item \label{enum:primaryfiltration} $M^{(j-1)} :_R m_j = (x_k ; x_k \not
\in G_j)$.
\item \label{enum:BasisAndprimaryfiltration} The set $\cup_{i=1}^j \{ u m_i
: u\, \text{is a monomial in}\, \Bbbk[G_i]\}$ is a $\Bbbk$-basis for
$M^{(j)}$.
\end{enumerate}
They imply that $\mathcal S$ is a Stanley filtration of $M$.}

Let $j=1$. We will show $(0:_Rm_1) = (x_k ; x_k \not \in G_1)$. For all
monomials $u \in \Bbbk[G_1]$, $um_1 \neq 0$, from the definition of the
decomposition. Therefore we must show that $x_l m_1 =0$ for any $x_l \not
\in G_1$.  Let $x_l \not \in G_1$.  Then $(\deg x_l m_1)^+$ is
square-free, and by \eqref{add1}, $x_lm_1 \in \mathrm{span}_\Bbbk
\{m_1,\dots,m_p\}$. However, from the choice of $m_1$, we see that $x_lm_1
= 0$. Therefore $(0:_Rm_1) = (x_k ; k \not \in G_1)$
proving~\eqref{enum:primaryfiltration}.  Note
that~\eqref{enum:BasisAndprimaryfiltration} follows immediately.

Now assume that $j>1$ and that the assertion is known for all $i < j$.
We first show \eqref{enum:primaryfiltration}.
Let
$u$ be a monomial in $\Bbbk[G_j]$. By the
statement~\eqref{enum:BasisAndprimaryfiltration} for $j-1$, the set
$\cup_{i=1}^{j-1} \{
vm_i : v\, \text{is a monomial in}\, \Bbbk[G_i]\}$ is a $\Bbbk$-basis for
$M^{(j-1)}$. Since $um_j$ is an element of the basis of $M$
coming from the Stanley decomposition, $um_j$ is not in the $\Bbbk$-linear
span of $\cup_{i=1}^{j-1} \{ vm_i : v\, \text{is a monomial in}\,
\Bbbk[G_i]\}$, \textit{i.e.}, $um_j \not \in M^{(j-1)}$.
It remains to prove that $x_l m_j \in M^{(j-1)}$ for any $x_l \not \in G_j$.
Let $x_l \not \in G_j$.
Since $(\deg x_lm_j)^+$ is square-free, it follows, from~\eqref{add1} and
the ordering of the $(m_i, G_i)$, that
\[
x_l m_j \in \mathrm{span}_\Bbbk\{m_i : 1 \leq i \leq p,\ \deg m_i > \deg
m_j\} \subseteq \mathrm{span}_\Bbbk\{m_1,\dots,m_{j-1}\}.
\]
Therefore $x_l m_j \in M^{(j-1)}$, proving the
statement~\eqref{enum:primaryfiltration} for $j$.

From~\eqref{enum:primaryfiltration}, we see that the following sequence is
exact:
\begin{equation}
\label{add2}
0  \longrightarrow M^{(j-1)} \longrightarrow M^{(j)} \longrightarrow
\Bbbk[G_j]m_j \longrightarrow 0.
\end{equation}
Now statement~\eqref{enum:BasisAndprimaryfiltration} for $j$ follows from
the induction hypothesis.

The assertion about regularity is
essentially~\cite{MacSmiUnifBdsMultgrReg05}*{Theorem~4.1}, but we give a
quick proof here. We will show that $\reg M^{(j)} \leq \max\{\totdeg{\deg
m_i} : 1 \leq i \leq j\}$ for all $1 \leq j \leq p$. It holds for $j=1$.
For $j > 1$, it follows from~\cite{eiscommalg}*{Corollary~20.19} and
the exact sequence~\eqref{add2} that $\reg M^{(j)} \leq \max \{\reg
M^{(j-1)}, \totdeg{\deg m_j}\}$; induction completes the proof.
\end{proof}

%Finally, we recall the basics of local cohomology,
Finally, we recall some basics of local cohomology,
following~\cite{BrHe:CM}*{Sections~3.5--3.6}.
Let  $\check C^\bullet$ be the \v Cech complex on
$x_1, \ldots, x_n$; the term at the $i$th cohomological degree is
\[
\check C^i = \bigoplus_{\Lambda \subseteq [n], |\Lambda| = i}
R_{x_\Lambda} 
\]
where $R_{x_\Lambda}$ denotes inverting the monomial $x_\Lambda$. Note that
$\check C^\bullet$ is a complex of $\ints^n$-graded $R$-modules, with
differentials of degree $0$.  For a
finitely generated $R$-module $M$, we set $\check C^\bullet(M) = 
\check C^\bullet\otimes_R (M)$. Then $\loco^i_{\mathfrak m}(M) = 
\homology^i(\check C^\bullet(M))$.

\begin{defn}
\label{thm:restrCechComplex}
Let $F \subseteq [n]$. We define $\check C_F^\bullet$ to be the subcomplex
of $\check C^\bullet$ obtained by setting
\[
\check C_F^i = 
\begin{cases}
0, & \text{if}\, i < |F|, \\
\bigoplus\limits_{\substack{F \subseteq \Lambda \subseteq [n] \\ |\Lambda| = i}}
R_{x_\Lambda}, & \text{otherwise}.
\end{cases}
\]
\end{defn}

\begin{lemma}
\label{thm:restrCechCx}
Let $I$ be a monomial ideal. Let $F \subseteq [n]$ and $\mathbf a \in
\ints^n$ be such that $\supp(\mathbf a^-) = F$. Then $\loco^i_{\mathfrak
m}(R/I)_{\mathbf a} = \homology^i(\check C_F^\bullet \otimes_R
(R/I))_{\mathbf a}$.
\end{lemma}

\begin{proof}
This argument is used implicitly in the proof of~\cite{TakaGCM05}*{Theorem~1}.
Since $\loco^i_{\mathfrak m}(R/I)_{\mathbf a} = \homology^i\left(( \check
C^\bullet(R/I))_{\mathbf a}\right)$, it suffices to show that $(\check
C^\bullet(R/I))_{\mathbf a} = (\check C_F^\bullet \otimes_R
(R/I))_{\mathbf a}$. This, in turn, stems from the fact that for all $1 \leq j \leq
n$, $\check C_F^j \otimes_R (R/I)$ consists precisely of the direct
summands of $\check C^j(R/I)$ that are non-zero in the multidegree $\mathbf
a$.
\end{proof}

\section{Proof of the main theorem}
\label{sec:regAndDim}

\begin{lemma}
\label{thm:multbyxjonExt}
Let $I \subset R$ be a monomial ideal.
Let $\mathbf a \in \ints^n$ and $j \in \supp(\mathbf
a^+)$. Then the multiplication map 
\[
x_j : \Ext^i_R(R/I, \omega_R)_{\mathbf a}
\longrightarrow \Ext^i_R(R/I, \omega_R)_{\mathbf a + \mathbf e_j}
\]
is bijective.
\end{lemma}

\begin{proof}
We first claim that the multiplication map
\[
x_j: \loco^{n-i}_{\mathfrak m}(R/I)_{-\mathbf a -\mathbf e_j} \longrightarrow
\loco^{n-i}_{\mathfrak m}(R/I)_{-\mathbf a}
\]
is bijective. By local
duality~\cite{BrHe:CM}*{Theorem~3.6.19}, this map is the Matlis dual of the
multiplication by $x_j$ on
$\Ext^i_R(R/I, \omega_R)_{\mathbf a}$; hence, it suffices to prove the
claim.

Set $F = \supp(\mathbf a^+)$. Note that
$\supp(\mathbf a^+ + \mathbf e_j)=F$. For all $i$, $x_j$ acts as a unit on 
$\check C_F^i$. Therefore the homomorphism of complexes
$\check C_F^\bullet \otimes_R (R/I) \to \check C_F^\bullet \otimes_R (R/I)$
induced by the multiplication map $x_j : \check C_F^i \otimes_R (R/I) \to
\check C_F^i \otimes_R (R/I)$ is an isomorphism.
The claim now follows from Lemma~\ref{thm:restrCechCx}, which implies that
$\loco^i_{\mathfrak m}(R/I)_{-\mathbf a -\mathbf e_j} = \homology^i(\check
C_F^\bullet \otimes_R (R/I))_{-\mathbf a - \mathbf e_j}$ and 
$\loco^i_{\mathfrak m}(R/I)_{-\mathbf a} = \homology^i(\check C_F^\bullet
\otimes_R (R/I))_{-\mathbf a}$. 
\end{proof}

The above lemma says that if $I$ is a monomial ideal then $\Ext^i_R(R/I,
\omega_R)$ is a $(1,1,\dots,1)$-determined module, in the sense
of~\cite{MillAlexDuality00}*{Definition~2.1}.

\begin{proof}[Proof of Theorem~\ref{thm:regIsLessThanDim}]
For $F \subseteq [n]$, let $\mathcal M^i_F$ be a multigraded
$\Bbbk$-basis for 
\[
\bigoplus_{\substack{\mathbf a \in \naturals^n\\ \supp(\mathbf
a) \cap F = \varnothing}} \Ext^i_R(R/I, \omega_R)_{\mathbf e_F- \mathbf a}.
\]
Let $\mathcal S_i = \{(m, F) : F \subseteq [n]\; \text{and}\; m \in
\mathcal M^i_F\}$. Then it follows from Lemma~\ref{thm:multbyxjonExt}
that $\mathcal S_i$ is a Stanley decomposition of
$\Ext^i_R(R/I, \omega_R)$.
In particular, 
$$\dim \Ext^i(R/I, \omega_R) = \max \{ |F| : \mathcal M^i_F
\neq \varnothing\}.$$
By the construction of $\mathcal M_F^i$,
this Stanley decomposition satisfies the assumption of Proposition~\ref{thm:stanleyDecompAndFiltr}.
Therefore
\begin{align*}
\reg \left(\Ext^i_R(R/I, \omega_R) \right)
& \leq \max_{F \subseteq [n]} \{ \max \{ \totdeg{\deg m} : m \in \mathcal
M^i_F\}\}  \\
& \leq \max_{F \subseteq [n]} \{|F| : \mathcal M^i_F \neq \varnothing\} \\
& = \dim \Ext^i_R(R/I, \omega_R),
\end{align*}
as desired.
(The second inequality follows from the fact that,
for any $u \in \mathcal M_F^i$,
one has $\totdeg{\deg u}=|F|-\totdeg{(\deg u)^-}$.)
\end{proof}

We remark that, using~\cite{TakaGCM05}*{Theorem~1} and local duality, one can
determine whether $\mathcal M^i_F \neq \varnothing$ from certain
subcomplexes of the Stanley-Reisner complex of the radical $\sqrt{I}$ of
$I$.\medskip

\noindent
\textbf{Acknowledgments.}
The authors thank B.~Ulrich for helpful comments.  This paper was written
when the second author was visiting Purdue University in September 2009. He
would like to thank his host, G.~Caviglia, for his hospitality.

%\bibliography{kummini}
\def\cfudot#1{\ifmmode\setbox7\hbox{$\accent"5E#1$}\else
  \setbox7\hbox{\accent"5E#1}\penalty 10000\relax\fi\raise 1\ht7
  \hbox{\raise.1ex\hbox to 1\wd7{\hss.\hss}}\penalty 10000 \hskip-1\wd7\penalty
  10000\box7}

\end{document}